\documentclass[12pt]{amsart}

\newtheorem{theorem}{Theorem}[section]
\newtheorem{proposition}[theorem]{Proposition}

\numberwithin{equation}{section}

\begin{document}

\baselineskip=15pt

\title[Rank one connections on abelian varieties, II]{Rank
one connections on abelian varieties, II}

\author[I. Biswas]{Indranil Biswas}

\address{School of Mathematics, Tata Institute of Fundamental
Research, Homi Bhabha Road, Bombay 400005, India}

\email{indranil@math.tifr.res.in}

\author[J. Hurtubise]{Jacques Hurtubise}

\address{Department of Mathematics, McGill University, Burnside
Hall, 805 Sherbrooke St. W., Montreal, Que. H3A 2K6, Canada}

\email{jacques.hurtubise@mcgill.ca}

\author[A. K. Raina]{A. K. Raina}

\address{Theoretical Physics Department, Tata Institute of
Fundamental Research, Homi Bhabha Road, Bombay 400005, India}

\email{raina@tifr.res.in}



\date{}

\begin{abstract}
Given a holomorphic line bundle $L$ on a compact complex torus 
$A$, there are two naturally associated holomorphic 
$\Omega_A$--torsors over $A$: one is 
constructed from the Atiyah exact sequence for $L$, and the 
other is constructed using the line bundle $(p^*_1 L^*)\otimes 
(\alpha^*L)$, where $\alpha$ is the addition map on $A\times A$, and $p_1$
is the projection of $A\times A$ to the first factor. In 
\cite{BHR}, it was shown that these two torsors are isomorphic.
The aim here is to produce a canonical isomorphism between them
through an explicit construction.
\end{abstract}

\maketitle

\section{Introduction}

Let $A$ be a complex abelian variety and $L$ a holomorphic line 
bundle over $A$. The sheaf of holomorphic connections on $L$
defines a torsor ${\mathcal C}_L$ on $A$ for the holomorphic cotangent
bundle $\Omega_A$.

Let $\alpha\, , p_1\, ,p_2 \, :\, A\times A\, \longrightarrow\, 
A$ be the addition map and the projections respectively. The 
holomorphic line bundle
$$
{\mathcal L}\, :=\, (p^*_1 L^*)\otimes (\alpha^*L)\,
\longrightarrow\,A\times A\,\stackrel{p_2}{\longrightarrow}\, A
$$
will be considered as a holomorphic family, parametrized by 
$A$, of topologically trivial holomorphic line bundles on $A$.
We have an $\Omega_A$--torsor
${\mathcal Z}_L\, \longrightarrow\, A$
whose holomorphic sections over any open subset $U\, \subset\,
A$ are the holomorphic families of relative holomorphic 
connections on ${\mathcal L}\vert_{A\times U}$.

In \cite{BHR} it was crucially used that the two torsors
${\mathcal C}_L$ and ${\mathcal Z}_L$ are holomorphically, 
or equivalently, algebraically, isomorphic (see \cite[Proposition 
2.1]{BHR}). The proof of Proposition 2.1 of \cite{BHR} was 
carried out by comparing the cohomological invariants associated 
to the torsors.

Our aim here is to give an explicit construction of a 
holomorphic isomorphism between the two torsors. The isomorphism 
is canonical in the sense that its construction does require 
making any choices.

We work with a compact complex torus; we do not need $A$ to be
algebraic.

\section{A criterion for isomorphism of torsors}

Let $M$ be a connected complex manifold. Let $\mathcal V$ be a 
holomorphic vector bundle over $M$.

A $\mathcal V$--\textit{torsor} on $M$ is a holomorphic fiber 
bundle
$p\, :\, Z\, \longrightarrow\, X$ and a holomorphic map from
the fiber product
$$
\varphi\, :\, Z\times_M {\mathcal V}\, \longrightarrow\, Z
$$
such that
\begin{enumerate}
\item $p\circ\varphi\,=\, p\circ p_Z$, where $p_Z$ is the
natural projection of $Z\times_M \mathcal V$ to $Z$,

\item the map $Z\times_M {\mathcal V}\, \longrightarrow\, 
Z\times_M Z$
defined by $p_Z\times \varphi$ is an isomorphism,

\item $\varphi(\varphi(z\, ,v)\, ,w)\,=\, \varphi(z\, ,v+w)$.
\end{enumerate}

A $\mathcal V$--torsor $(Z\, ,p\, ,\varphi)$ is called 
\textit{trivializable} if there is a holomorphic isomorphism 
$$
\beta\, : {\mathcal V}\,\longrightarrow\, Z
$$
such that $p\circ\beta$ is the natural projection of $\mathcal V$ to $M$, and
$$
\beta^{-1}\circ \varphi\circ (\beta\times\text{Id}_{\mathcal 
V})\, :\, {\mathcal V}\times_M {\mathcal V}\, \longrightarrow\,
{\mathcal V}
$$
is the fiberwise addition homomorphism on ${\mathcal V}\times_M {\mathcal V}$. A
holomorphic isomorphism $\beta$ satisfying 
the above conditions is called a \textit{trivialization} of $Z$.

Any $\mathcal V$--torsor $Z$ has a $C^\infty$ section $M\,\longrightarrow 
\, Z$ because the fibers of $Z$ are contractible. It has a 
holomorphic section if and only if it is trivializable.

Take a $\mathcal V$--torsor $(Z\, ,p\, ,\varphi)$. Let
$$
\sigma\,:\, M\, \longrightarrow\, Z
$$
be a $C^\infty$ section, so $p\circ\sigma\,=\, \text{Id}_M$.
Let
$$
d\sigma\, :\, T^{\mathbb R}M\, \longrightarrow\, 
\sigma^* T^{\mathbb R} Z
$$
be the differential of $\sigma$, where $T^{\mathbb R}$ is the 
real tangent bundle. The almost complex structures
on $M$ and $Z$ will be denoted by $J_M$ and $J_Z$ respectively.
Let
$$
\widetilde{\sigma}\, :\, T^{\mathbb R}M\, \longrightarrow\, 
\sigma^* T^{\mathbb R} Z\, ,~\, ~ \, v\, \longmapsto\,
d\sigma (J_M(v)) - J_Z(d\sigma (v))
$$
be the obstruction for $\sigma$ to be holomorphic. Since the 
projection $p$ is holomorphic, we have
$$
dp (\widetilde{\sigma}(v))\,=\, dp(d\sigma (J_M(v))) - dp
(J_Z(d\sigma (v)))\,=\, J_M(v) - J_M(dp(d\sigma(v)))\,=\, 0\, ,
$$
where $dp$ is the differential of $p$. Hence 
$\widetilde{\sigma}(v)$ is an element of ${\mathcal V}_x$ if
$v\, \in\, T^{\mathbb R}_xM$ (the vertical tangent subbundle
of $T^{\mathbb R}Z$ for
$p$ is identified with $p^*{\mathcal V}$). Note that
$\widetilde{\sigma}(J_M(v))\,=\,-J_Z(\widetilde{\sigma}(v))$.
Define
\begin{equation}\label{e2}
\widehat{\sigma}\, :\, T^{0,1}M\, \longrightarrow\, {\mathcal 
V}\, ,\, ~\, ~ \, v+\sqrt{-1}\cdot J_M(v)\, \longmapsto\, 
\widetilde{\sigma}(v)\, .
\end{equation}
So $\widehat{\sigma}$ is a smooth $(0\, ,1)$--form with values 
in $\mathcal V$.
Clearly, $\widehat{\sigma}\,=\,0$ if and only if $\sigma$ is 
holomorphic. Equivalently, $\widehat{\sigma}\,=\,0$ if and only 
if $\sigma$ is a trivialization of $Z$.

Let $(Z_1\, ,p_1\, ,\varphi_1)$ and $(Z_2\, ,p_2\, ,\varphi_2)$
be two $\mathcal V$--torsors. Let $\sigma$ and $\tau$ be 
$C^\infty$ 
sections of $Z_1$ and $Z_2$ respectively. We have a unique $C^\infty$ 
isomorphism of $\mathcal V$--torsors
\begin{equation}\label{e3}
\gamma\, :\, Z_1\, \longrightarrow\, Z_2\, ,~\,~\,
\gamma\circ \varphi_1(\sigma(x)\, ,v)\,=\, \varphi_2(\tau(x)\, 
,v) \, ,~\, x\,\in\, M\, , v\,\in\, {\mathcal V}_x\, .
\end{equation}
So, $\gamma\circ\sigma\,=\, \tau$.

\begin{proposition}\label{prop1}
If $\widehat{\sigma}\,=\, \widehat{\tau}$ (constructed as
in \eqref{e2}), then $\gamma$ in \eqref{e3} is holomorphic.
\end{proposition}

\begin{proof}
Let
$$
q\, :\, H\, \longrightarrow\, M
$$
be the holomorphic fiber bundle whose fiber over any $x\, \in\, 
M$ is the space of all isomorphisms $\phi\, :\, (Z_1)_x\,
\longrightarrow\, (Z_2)_x$ such that $\phi\circ\varphi_1 
(z\, ,v)\,=\, \varphi_2(\phi(z)\, ,v)$ 
for all $v\, \in\, {\mathcal V}_x$. Note that we have a map
$$
\varphi_x\, :\, H_x\times {\mathcal V}_x\, \longrightarrow\, 
H_x\, ,~\,~\,
\varphi_x(\phi\, ,v)(z)\,=\, \varphi_2(\phi(z)\, ,v)\,=\,
\phi\circ \varphi_1(z\, ,v)\, .
$$
There is a complex structure on $H$ uniquely determined by 
the condition that a section of $H$
defined over any open subset $U\, \subset\, M$ is holomorphic
if the corresponding map $Z_1\vert_U \, \longrightarrow\,
Z_2\vert_U$ is holomorphic.
The triple $(H\, ,q\, ,\varphi)$ is a $\mathcal V$--torsor, 
where $\varphi\vert_{H_x\times{\mathcal V}_x}\,:=\, \varphi_x$.

The map $\gamma$ in \eqref{e3} defines a $C^\infty$ section $M
\, \longrightarrow\, H$, which will also be denoted by $\gamma$. 
It is straightforward to check that $\widehat{\gamma}\,=\, 
\widehat{\tau}-\widehat{\sigma}$. Therefore, if $\widehat{\tau}
\,=\, \widehat{\sigma}$, then $\widehat{\gamma} \,=\, 0$. As 
noted before, $\widehat{\gamma}\,=\, 0$ if and only if $\gamma$ 
is holomorphic.
\end{proof}

\section{Line bundles on a complex torus}

Let $A$ be a compact complex torus. Let
\begin{equation}\label{a}
\alpha\, :\, A\times A\, \longrightarrow\, A
\end{equation}
be the addition map. Let $p_i\, :\, A\times A\, 
\longrightarrow\, A$, $i\,=\, 1\, ,2$, be the projection
to the $i$-th factor. Take a holomorphic line bundle $L$ over 
$A$. The holomorphic line bundle
\begin{equation}\label{l}
{\mathcal L}\, :=\, (p^*_1 L^*)\otimes (\alpha^*L)\,
\longrightarrow\,A\times A\,\stackrel{p_2}{\longrightarrow}\, A
\end{equation}
will be considered as a holomorphic family of holomorphic line bundles on 
$A$ parametrized by $A$. For any $x\, \in\, A$, consider
the holomorphic line bundle
\begin{equation}\label{x}
{\mathcal L}^x\, :=\, {\mathcal L}\vert_{A\times\{x\}} 
\longrightarrow\, A\, .
\end{equation}
It is topologically trivial, so ${\mathcal L}^x$ admits 
holomorphic connections.

The holomorphic cotangent bundle of $A$ will be denoted by
$\Omega_A$. Define $V\,:=\, H^0(A,\, \Omega_A)$, and let
\begin{equation}\label{v}
{\mathcal V}\,:=\, A\times V\, \longrightarrow\, A
\end{equation}
be the trivial holomorphic vector bundle. The space of all holomorphic 
connections on ${\mathcal L}^x$ is 
an affine space for $V$. We have a $\mathcal V$--torsor
\begin{equation}\label{z}
p_Z\, :\, {\mathcal Z}_L\, \longrightarrow\, A
\end{equation}
whose fiber over any point $x\, \in\, X$ is the space of all
holomorphic connections on ${\mathcal L}^x$. A holomorphic 
section of ${\mathcal Z}_L$ defined over an open subset $U\, 
\subset\, A$ is a holomorphic family of relative holomorphic 
connections on ${\mathcal L}\vert_{A\times U}$. This condition 
determines uniquely the complex structure of ${\mathcal Z}_L$.

Let
$$
0\, \longrightarrow\, {\mathcal O}_A \, \longrightarrow\,
\text{At}(L) \, \longrightarrow\, TA \, \longrightarrow\, 0
$$
be the Atiyah exact sequence for $L$ (see \cite{At}); here $TA$
is the holomorphic tangent bundle. Let
\begin{equation}\label{at}
0\, \longrightarrow\, \Omega_A \, \longrightarrow\, 
\text{At}(L)^* \, \stackrel{\lambda}{\longrightarrow}\,{\mathcal 
O}_A \, \longrightarrow\, 0
\end{equation}
be the dual of the Atiyah exact sequence. Let $1_A$ be the
section of ${\mathcal O}_A$ given by the constant function
$1$. From \eqref{at} it follows that
\begin{equation}\label{c}
p_C\, :\, {\mathcal C}_L\, :=\, \lambda^{-1}(1_A)\, 
\longrightarrow\, A
\end{equation}
is an $\Omega_A$-torsor. Note that the vector bundle $\Omega_A$ 
is canonically identified with $\mathcal V$ (defined in 
\eqref{v}) using the evaluation map on sections. 
Therefore, ${\mathcal C}_L$ is a $\mathcal V$--torsor on $A$.

\begin{theorem}\label{thm1}
The two $\mathcal V$--torsors ${\mathcal Z}_L$ and ${\mathcal 
C}_L$, constructed in \eqref{z} and \eqref{c} respectively, are 
canonically holomorphically isomorphic.
\end{theorem}

\begin{proof}
We will show that both ${\mathcal Z}_L$ and ${\mathcal C}_L$
have tautological $C^\infty$ sections.

There is a unique translation invariant, with respect to
$\alpha$ in \eqref{a}, $(1\, ,1)$--form 
$\omega$ on $A$ representing $c_1(L)\, \in\, H^2(A,\, {\mathbb 
Q})$ (any translation invariant form on $A$ is closed). There is a 
unique unitary complex connection $\nabla_L$ on 
$L$ such that the curvature of $\nabla_L$ is $\omega$. The 
hermitian structure on $L$ for $\nabla_L$ is determined uniquely 
up to multiplication by a constant positive real number. Since 
complex connections on $L$ are $C^\infty$ splittings of 
\eqref{at}, the 
connection $\nabla_L$ defines a $C^\infty$ section of the
$\mathcal V$--torsor ${\mathcal C}_L$ in \eqref{c}. Let
\begin{equation}\label{si}
\sigma\, :\, A\, \longrightarrow\, {\mathcal C}_L
\end{equation}
be this $C^\infty$ section.

Consider the holomorphic line bundle ${\mathcal L}$ in 
\eqref{l}. The connection $\nabla_L$ on $L$ pulls back to
connections on both $p^*_1 L$ and $\alpha^*L$, and the
connection on $p^*_1 L$ produces a connection on
$p^*_1 L^*$. Therefore, we get a unitary complex connection on 
${\mathcal L}$ from $\nabla_L$; this connection on ${\mathcal L}$ will be 
denoted by $\nabla_{\mathcal L}$.

For any point $x\, \in\, A$, let $\nabla^x_{\mathcal L}$ be the
restriction of $\nabla_{\mathcal L}$ to the line bundle ${\mathcal L}^x$
defined in \eqref{x}. Since the curvature of the connection 
$\nabla_L$ is translation invariant, the curvature of 
$\nabla^x_{\mathcal L}$ vanishes identically. Hence 
$\nabla^x_{\mathcal L}$ is a holomorphic connection on
${\mathcal L}^x$. Consequently, we get a $C^\infty$ section of 
the $\mathcal V$--torsor ${\mathcal Z}_L$ in \eqref{z}
\begin{equation}\label{ta}
\tau\, :\, A\, \longrightarrow\, {\mathcal Z}_L\, ,~\, ~\,
x\, \longmapsto\, \nabla^x_{\mathcal L}\, .
\end{equation}

In view of Proposition \eqref{prop1}, to prove the theorem it 
suffices to show that $\widehat{\sigma}\,=\, \widehat{\tau}$, 
where $\tau$ and $\sigma$ are constructed in \eqref{ta} and
\eqref{si} respectively. Note that a smooth
$(0\, ,1)$--form on $M$ with values in $V\, :=\, H^0(A,\, 
\Omega_A)\,=\, \Omega_A$ is a $(1\, ,1)$--form on $A$.

As before, $\omega$ denotes the curvature of $\nabla_L$.
It is standard that
\begin{equation}\label{i1}
\widehat{\sigma}\,=\, \omega
\end{equation}
(it is a general expression of the curvature of a connection in 
terms of the splitting of the Atiyah exact sequence for the 
connection).

Consider the connection $\nabla_{\mathcal L}$ on $\mathcal L$
constructed above. Since the curvature of $\nabla_L$ is 
$\omega$, it follows immediately that the curvature of
the connections on $p^*_1 L^*$ and $\alpha^*L$ are
$-p^*_1\omega$ and $\alpha^*\omega$ respectively. Hence
the curvature of $\nabla_{\mathcal L}$ is
$$
{\mathcal K}(\nabla_{\mathcal L})\,=\, \alpha^*\omega - 
p^*_1\omega
$$
(see \eqref{l}). For any $y\, \in\, A$, let
$$f_y\, :\, A\, \longrightarrow\, A\times A$$ be the section of
$p_2$ defined by $x\, \longrightarrow\, (y\, ,x)$. We have
\begin{equation}\label{i}
f^*_y{\mathcal K}(\nabla_{\mathcal L})\,=\, f^*_y\alpha^*\omega 
-f^*_y p^*_1\omega\, =\, \omega
\end{equation}
because $\omega$ is translation invariant. From \eqref{i} it
follows that 
\begin{equation}\label{i2}
\widehat{\tau}\,=\, \omega\, ,
\end{equation}
where $\tau$ is constructed in \eqref{ta} (see \eqref{e2}); the 
curvature ${\mathcal K}(\nabla_{\mathcal L})$ is the obstruction 
for the $C^\infty$ family of holomorphic connections 
$\nabla^x_{\mathcal L}$ to be holomorphic. From
\eqref{i1} and \eqref{i2} we conclude that $\widehat{\tau}\,=\,
\widehat{\sigma}$. Therefore, the proof of the theorem is complete by
Proposition \ref{prop1}.
\end{proof}

Theorem \ref{thm1} immediately implies Proposition 2.1 of
\cite{BHR} by giving the isomorphism $\eta$ in the proposition
explicitly. The factor $-1$ in Proposition 2.1 of \cite{BHR}
arises because ${\mathcal C}_L$ and ${\mathcal Z}_{L^*}$ is being
compared (instead of ${\mathcal C}_L$ and ${\mathcal Z}_L$).
Note that there is a natural holomorphic isomorphism
$\delta\, :\, {\mathcal Z}_L\, \longrightarrow\, {\mathcal Z}_{
L^*}$ between the total spaces such that $\delta(z+v)\,=\,
\delta(z)-v$ for all $v\, \in\, \mathcal V$; this is because
there is a natural bijection between the connections on a line
bundle $\zeta$ and its dual $\zeta^*$. For the same reason,
there is a holomorphic isomorphism $\delta'\, :\, {\mathcal C}_L
\, \longrightarrow\, {\mathcal C}_{L^*}$ between the total
spaces such that $\delta'(z+v)\,=\, \delta'(z)-v$ for all
$v\, \in\, \mathcal V$.

\medskip
\noindent
\textbf{Acknowledgements.}\, The first two authors thank
the Issac Newton Institute for hospitality while the work
was carried out.


\end{document}